\documentclass[12pt, reqno]{amsart}

\usepackage{latexsym}
\usepackage{amsmath}
\usepackage{amsfonts}
\usepackage{amssymb}
\usepackage{graphicx}
\usepackage{comment}
\usepackage{enumitem}

\newtheorem{theorem}{Theorem}[section]
\newtheorem{lemma}[theorem]{Lemma}
\newtheorem*{theorem*}{Theorem}
\newtheorem{corollary}[theorem]{Corollary}

\newtheorem* {question*}{Question}

\newtheorem{sublemma}{}[theorem]
\newtheorem{conjecture}[theorem]{Conjecture}

\theoremstyle{definition}

\theoremstyle{remark}

\numberwithin{equation}{section}

\newcommand{\si}{{\rm si}}
\newcommand{\cl}{{\rm cl}}

\begin{document}

\title[$GF(q)$-Chordal Matroids II]{Chordal Matroids arising from Generalized parallel connections II}

\author{James Dylan Douthitt}
\address{Mathematics Department\\
Louisiana State University\\
Baton Rouge, Louisiana}
\email{jdouth5@lsu.edu}

\author{James Oxley}
\address{Mathematics Department\\
Louisiana State University\\
Baton Rouge, Louisiana}
\email{oxley@math.lsu.edu}

\subjclass{05B35, 05C75}
\date{\today}
\keywords{chordal graph, parallel connection, projective geometry}

\begin{abstract}
In 1961, Dirac showed that chordal graphs are exactly the graphs that can be constructed from complete graphs by a sequence of clique-sums. In an earlier paper, by analogy with Dirac's result, we introduced the class of $GF(q)$-chordal matroids as those matroids that can be constructed from projective geometries over $GF(q)$ by a sequence of generalized parallel connections across projective geometries over $GF(q)$. Our main result showed that when $q=2$, such matroids have no induced minor in $\{M(C_4),M(K_4)\}$. In this paper, we show that the class of $GF(2)$-chordal matroids coincides with the class of binary matroids that have none of $M(K_4)$, $M^*(K_{3,3})$, or $M(C_n)$ for $n\geq 4$ as a flat. We also show that $GF(q)$-chordal matroids can be characterized by an analogous result to Rose's 1970 characterization of chordal graphs as those that have a perfect elimination ordering of vertices. This post-publication version of the paper corrects an error in the published version. This error was pointed out by Tuong Le and Chayim Lowen.
\end{abstract}

\maketitle

\section{Introduction}
\label{introduction}
The notation and terminology in this paper will follow \cite{diestel} for graphs and \cite{ox1} for matroids. Unless stated otherwise, all graphs and matroids considered here are simple. Thus every contraction of a set from a matroid is immediately followed by the simplification of the resulting matroid. We will also assume all matroids are binary unless otherwise specified. Following Cordovil, Forge, and Klein \cite{cordovil}, we define a simple or non-simple matroid $M$ to be \textit{chordal} if, for each circuit $D$ that has at least four elements, there are circuits $D_1$ and $D_2$ and an element $e$ such that $D_1\cap D_2 =\{e\}$ and $D=(D_1 \cup D_2)-e$. Therefore, a simple binary matroid is chordal precisely when it has no member of $\{M(C_n):n\geq 4\}$ as a flat, where $C_n$ is the $n$-edge cycle. 
In Section~\ref{ChordalMatroids}, we prove an assertion made in \cite{douthox} that the class of such matroids coincides with the class of binary matroids with no $M(C_4)$ as an induced minor, where an \textit{induced minor} of a matroid $M$ is any matroid that can be obtained from $M$ by a sequence of contractions and restrictions to flats. This matroid notion is analogous to a graph notion, an induced minor of a graph $G$ being a graph that can be obtained from $G$ by a sequence of vertex deletions and edge contractions.

For a prime power $q$, we denote the projective geometry $PG(r-1,q)$ by $P_r$ when context makes the field clear. A matroid is \textit{$GF(q)$-chordal} if it can be obtained by repeated generalized parallel connections across projective geometries over $GF(q)$ starting with projective geometries over $GF(q)$. In Section \ref{GF(2)-chord}, we prove the next theorem, which is the main result of the paper, and we give an analogous result for $q>2$. The equivalence of (i) and (ii) was shown in \cite{douthox}.
\begin{theorem}\label{noC4K4}
    The following are equivalent for a binary matroid $M$.
    \begin{enumerate}[label=\normalfont(\roman*)]
        \item $M$ is $GF(2)$-chordal.
        \item $M$ has no member of $\{M(C_4),M(K_4)\}$ as an induced minor.
        \item $M$ has no member of $\{M(C_n):n\geq 4\}\cup \{M(K_4),M^*(K_{3,3})\}$ as an induced restriction.
    \end{enumerate}
\end{theorem}

In Section~\ref{GF(2)-chord}, we also prove the following analog of Theorem~\ref{noC4K4} for all other primes.

\begin{theorem}\label{gf(p)-chord}
    For each prime $p>2$, the following are equivalent for a $GF(p)$-representable  matroid $M$.
    \begin{enumerate}[label=\normalfont(\roman*)]
        \item $M$ is $GF(p)$-chordal.
        \item $M$ has no member of $\{U_{2,k}:3\leq k\leq p\}$ as an induced minor.
        \item $M$ has no member of $\{U_{n,n+1}: n\geq 2\}\cup \{U_{2+t,k+t}:4\leq k \leq p \text{ and } 0\leq t\leq p+1-k\}$ as an induced restriction.
    \end{enumerate}
\end{theorem}

Chordal graphs have been characterized in several other ways apart from Dirac's \cite{dirac} description. A \textit{perfect elimination ordering of a graph} $G$ is an ordering of $V(G)$ such that, for every vertex $v$, the graph induced by $v$ and all of its neighbors that occur after $v$ in the ordering is a clique.  In 1970, Rose~\cite{rose} proved the following characterization.

\begin{theorem}\label{rosethm}
    A graph $G$ chordal if and only if $G$ has a perfect elimination ordering.
\end{theorem}

A \textit{perfect elimination ordering of cocircuits} of a matroid $M$ is a collection $C_1^*,C_2^*,\dots,C_r^*$ such that, for all $i$ in $[r]$, the set $C_i^*$ is a cocircuit of the matroid $M_i$, which is $M\backslash (C_1^*\cup C_2^*\cup \dots \cup C_{i-1}^*)$, and $M|\cl_{M_i}(C_i^*)$ is a projective geometry. In Section~\ref{Perfect Elimination Ordering}, we prove the following analog of Theorem~\ref{rosethm} for $GF(q)$-chordal matroids.

\begin{theorem}\label{perfectisgfq}
    A matroid $M$ is $GF(q)$-chordal if and only $M$ has a perfect elimination ordering of cocircuits.
\end{theorem}

\section{Binary chordal matroids}\label{ChordalMatroids}
In this section, we will show that the class of binary chordal matroids coincides with the class of binary matroids with no $M(C_4)$ as an induced minor. 
\begin{lemma}\label{circRes}
    Let $n$ be the size of a largest circuit that is an induced minor of a binary matroid $M$. Then $M$ has an $n$-element circuit as an induced restriction.
\end{lemma}

\begin{proof} 
    We may assume that, for some independent set $I$ of $M$, the matroid $\si(M/I)$ has an $n$-element circuit as a flat. If $|I|=0$, then the result holds. Assume the result holds for $|I|<k$, and let $|I|=k\geq 1$. Take $e\in I$. Then $\si((M/e)/(I-e))$ has an $n$-element circuit as a flat. Certainly a largest circuit that occurs as an induced minor of $\si(M/e)$ has $n$ elements. Thus, by the induction assumption, $\si(M/e)$ has, as an induced restriction, an $n$-element circuit $C$ where $C=\{e_1,e_2,\dots,e_n\}$. Then $C$ is a circuit of $M/e\backslash Y$ for some set $Y$. Thus $C$ or $C\cup e$ is a circuit of $M$. We may assume that no $n$-element circuit is an induced restriction of $M$. Now view $M$ as a restriction of the binary projective geometry $P_r$ where $r=r(M)$. For each $i$ in $[n]$, let $f_i$ be the third point on the projective line $\cl_{P_r}(\{e,e_i\})$. For each $i$ in $[n]$, the set $\{e_1,e_2,\dots,e_n\}-\{e_i\}$ is an independent set $I_i$ of $M/e$ and hence of $M$. Now $\cl_M(I_i)$ does not contain $e$ otherwise $I_i$ contains a circuit of $M/e$. Moreover, $\cl_M(I_i)$ does not contain $e_i$ as $M$ does not have an $n$-element circuit as an induced restriction. The projective flat $\cl_{P_r}(I_i)$ must meet $\{e,f_i,e_i\}$ in $P_r$, so this intersection is $f_i$. If $f_i$ is in $E(M)$, then $M$ has $I_i\cup f_i$ as an induced restriction that is an $n$-element circuit. Thus $\{f_1,f_2,\dots,f_n\}$ avoids $E(M)$. We deduce that $C\cup e$ is an $(n+1)$-element circuit of $M$ that is an induced minor of $M$, a contradiction.    
\end{proof}
In the next theorem, the equivalence of (i) and (iii) is an immediate consequence of the definition. In \cite[Lemma 3.8]{douthox}, we had asserted the equivalence of (i) and (ii). Our proof of this relies on Lemma~\ref{circRes} and is more subtle than we originally thought, so we have included it.
\begin{theorem}\label{noC4}
    The following are equivalent for a binary matroid $M$.
    \begin{enumerate}[label=\normalfont(\roman*)]
        \item\label{i} $M$ is chordal.
        \item\label{ii} $M$ has no $M(C_4)$ as an induced minor.
        \item $M$ has no member of $\{M(C_n):n\geq 4\}$ as an induced restriction.
    \end{enumerate}
\end{theorem}

\begin{proof}
    Clearly (iii) implies (ii). Now suppose that $M$ has $M(C_4)$ as an induced minor. Let $n$ be the size of a largest circuit that is an induced minor of $M$. Since $M$ has $M(C_4)$ as an induced minor, $n\geq 4$. Then, by Lemma~\ref{circRes}, $M$ has an $n$-element circuit as an induced restriction, that is, $M$ has $M(C_n)$ as an induced restriction for some $n\geq 4$.
\end{proof}

We now build toward a proposed constructive characterization of binary chordal matroids. In a matroid $M$, denote a vertical $k$-separation $(X,Y)$ by $(X,G,Y)$ where $G=\cl_M(X)\cap \cl_M(Y)$. The vertical $k$-separation $(X,Y)$ is \textit{exact} if $r(X)+r(Y)-r(M)=k-1$. For a rank-$r$ binary matroid $M$ that is viewed as a restriction of $P_r$, if $X\subseteq E(P_r)-E(M)$, we denote by $M+X$ the matroid $P_r|(E(M)\cup X)$.

\begin{lemma}\label{emptyGuts}
    For some $k\geq 2$, let $(X,G,Y)$ be an exact vertical $k$-separation of a binary matroid $M$ where $G=\emptyset$. Then $M$ has $M(C_4)$ as an induced minor.
\end{lemma}
\begin{proof}
    Let $r(M)=r$. Since $X$ spans $\cl_{P_r}(X)$, we may choose $C_0$ to be a smallest set contained in $X$ such that, for some $z$ in $\cl_{P_r}(X)\cap \cl_{P_r}(Y)$, the set $C_0\cup \{z\}$ is a circuit of $M|\cl_{P_r}(X)$. Since $(X,G,Y)$ is an exact vertical $k$-separation, such a point $z$ must exist. Let $M'=M|\cl_M(Y\cup C_0)$. Since $M'$ is simple and binary, it follows by the choice of $C_0$ that $\cl_{P_r}(C_0)\cap (\cl_{P_r}(X)\cap \cl_{P_r}(Y))=\{z\}$. We conclude that $M'$ decomposes as the 2-sum of $(M|\cl_M(Y))+\{z\}$ and $(M|\cl_M(C_0))+\{z\}$ at the basepoint $z$. 
    Let $a$ and $b$ be distinct elements of $C_0$ and let $N=M'/(C_0-\{a,b\})$. Then $N$ decomposes as the 2-sum of $(M|\cl_M(Y))+\{z\}$ and the triangle $\{a,b,z\}$ at the basepoint $z$. Let $D_0$ be a smallest set contained in $Y$ such that $D_0\cup \{z\}$ is a circuit of $(M|\cl_M(Y))+\{z\}$. Then the set $D_0\cup \{a,b\}$ is a circuit of $N$, and, by the choice of $D_0$, the set $D_0\cup \{a,b\}$ is a flat of $N$.
    Therefore, $N$ has the circuit $D_0\cup \{a,b\}$ as an induced restriction. By Theorem~\ref{noC4}, $N$ has $M(C_4)$ as an induced minor. We conclude that $M$ has $M(C_4)$ as an induced minor.
\end{proof}
\begin{lemma}\label{fullGuts}
    For some $k\geq 2$, suppose $(X,G,Y)$ is an exact vertical $k$-separation of a binary chordal matroid $M$. Then $r_M(G)=k-1$.
\end{lemma}
\begin{proof}
    Suppose $r(G)<k-1$. Then, for $k'=k-r(G)$, we see that $k'\geq 2$ and $M/G$ has an exact vertical $k'$-separation $(X',G',Y')$ with $G'=\emptyset$. By Lemma~\ref{emptyGuts}, $M$ has $M(C_4)$ as an induced minor, a contradiction to Theorem~\ref{noC4}. Therefore, $r(G)=k-1$. 
\end{proof}

In the published version of this paper, we had incorrectly asserted in Lemma 2.5 that if $G$ is a flat of a matroid $N$, and $G-g$ is a modular flat of $N/g$ for some $g$ in $G$, then $G$ is a modular flat of $N$. As pointed out by Tuong Le and Chayim Lowen in a private communication, our proof of this was wrong. Moreover, our assertion remains false even if $G-g$ is a modular flat of $N/g$ for all $g$ in $G$. To see this, let $N$ be $P_T (M(K_4),F_7)\backslash T$ where $T$ is a triangle, and let $G=E(F_7)-T$. Then $G$ is a non-modular flat of $N$, but $G-g$ is a modular flat of $N/g$ for all $g$ in $G$.

Based on the erroneous Lemma 2.5, we  also incorrectly asserted in Lemma 2.6 of the published paper that if $(X,G,Y)$ is an exact vertical $k$-separation of a binary chordal matroid and $k\geq 2$, then $G$ is a modular flat of $M|\cl(X)$ or of $M|\cl(Y)$. To see this is false, let $M = M({\mathcal W}_5)\backslash 0$ where the spokes of ${\mathcal W}_5$ are $1,3,5,7,$ and $9$, and the triangles are $\{1,2,3\}$, $\{3,4,5\}$, $\{5,6,7\}$, $\{7,8,9\}$, and $\{9,0,1\}$. Let $X=\{2,3,4,5,9\}$ and $Y=\{1,6,7,8\}$. Then 
$(X,Y)$ is an exact vertical $4$-separation of $M$, and 
$\cl(X)\cap \cl(Y)=\{1,2,3\}$. But the last  set is not a modular flat of $M|\cl(X)$ or of $M|\cl(Y)$.

The following conjecture was previously claimed as a theorem in the published version of our paper. We know of no counterexample to it. 

\begin{conjecture}
    All binary chordal matroids can be obtained by starting with round binary chordal matroids and repeatedly taking generalized parallel connections of two previously constructed matroids across a set that is a modular flat of one of them. 
\end{conjecture}

\section{$GF(q)$-chordal matroids}\label{GF(2)-chord}
The goal of this section is to prove Theorem~\ref{noC4K4}. Recall that a matroid is $GF(2)$-chordal if it can be built from binary projective geometries by repeated generalized parallel connections across projective geometries. In \cite[Theorem 1.3]{douthox}, we showed that the class of $GF(2)$-chordal matroids is closed under taking induced minors and therefore, the class is also closed under taking induced restrictions. For a binary matroid $M$, we may uniquely specify $M$ by describing its complement in the projective geometry $P_{r(M)}$. In general, Brylawski and Lucas \cite{brylawski-lucas} (see also \cite[Proposition 10.1.7]{ox1}) showed that the complement of any uniquely $GF(q)$-representable matroid $M$ is well defined in any projective geometry of rank at least $r(M)$. We will first examine the matroids that have $M(K_4)$ as an induced minor.

\begin{lemma}\label{10less}
    If $M$ is a simple binary matroid with $r(M)=4$ and $|E(M)|>9$, then either $M$ has $M(C_4)$ or $M(K_4)$ as an induced restriction, or $M$ does not contain $M(K_4)$ as an induced minor.
\end{lemma}
\begin{proof}
    Let $P_4$ be the rank-4 binary projective geometry. Suppose $T$ is a largest subset of $E(P_4)$ such that $P_4|T$ contains $M(K_4)$ as an induced minor, but $T$ does not contain one of $M(C_4)$ or $M(K_4)$ as an induced restriction. Certainly, $P_4$ does not contain $M(K_4)$ as an induced minor and so $|T|<15$. If $|T|\in \{13,14\}$, then $P_4|T$ certainly has $M(K_4)$ as an induced restriction. If $|T|=12$, then $P_4\backslash T$ is either $P_2$ or $U_{3,3}$. If $P_4\backslash T$ is $P_2$, then $P_4|T$ has $M(C_4)$ as an induced restriction. If $P_4\backslash T$ is $U_{3,3}$, then $P_4|T$ will again contain $M(K_4)$ as an induced restriction. If $|T|=11$, then $P_4\backslash T$ is $M(C_4)$, $U_{4,4}$, or $P_2\oplus U_{1,1}$. When $P_4\backslash T$ is $M(C_4)$, we see that $P_4|T\cong P_{P_2}(P_3,P_3)$, so it has no element whose contraction gives $M(K_4)$. If $P_4\backslash T$ is $U_{4,4}$, then $P_4|T$ has a plane with exactly six points, so $P_4|T$ has $M(K_4)$ as an induced restriction. If $P_4\backslash T$ is $P_2\oplus U_{1,1}$, then $P_4|T$ has $M(C_4)$ as an induced restriction. If $|T|=10$, then $P_4\backslash T$ is $M(K_4-e)$, $M(C_4)\oplus U_{1,1}$, or $P_2\oplus U_{2,2}$. In each case, $P_4|T$ has $M(K_4)$ as an induced restriction. 
\end{proof}
\begin{lemma}\label{K4incontraction}
    If $M/f$ has $M(K_4)$ as an induced restriction for some $f$ in $E(M)$, then $M$ has $M(C_4),M(K_4),$ or $M^*(K_{3,3})$ as an induced restriction.
\end{lemma}
\begin{proof}
    It suffices to show the result holds for $r(M)=4$ since we may restrict to a rank-4 flat $F$ such that $(M|F)/f\cong M(K_4)$. Assume $M$ does not have any of $M(C_4),M(K_4),$ or $M^*(K_{3,3})$ as an induced restriction. Since $M/f$ is isomorphic to $M(K_4)$, it follows that $|E(M)|\geq 7$, and, by Lemma~\ref{10less}, we have $|E(M)|\leq 9$. Certainly $M$ is connected. If $M$ is not 3-connected, then $M$ is isomorphic to the 2-sum or the parallel connection of $U_{2,3}$ and  $M(K_4)$. In the former case, $M$ has $M(C_4)$ as an induced restriction, and, in the latter case, $M$ has $M(K_4)$ as an induced restriction. Thus we may assume $M$ is 3-connected. Suppose $M$ is graphic. Then $M$ is obtained from $M(K_5)$ by deleting at most two edges. If two incident edges have been deleted from $K_5$, then $M$ is not 3-connected. Therefore $M$ is isomorphic to $M(K_5\backslash e)$ for some edge $e$, or $M$ is isomorphic to $M(\mathcal{W}_4)$. In the first case, $M$ has $M(K_4)$ as an induced restriction; in the second case, $M$ has $M(C_4)$ as an induced restriction.

    We may now assume that $M$ is not graphic. Then $M$ has one of $M^*(K_5)$, $M^*(K_{3,3})$, $F_7$, or $F_7^*$ as a minor by a theorem of Tutte~\cite{tutte} (see also \cite[Theorem 6.6.7]{ox1}). Since $|E(M)|\leq 9$, we have $M\not\cong M^*(K_5)$. If $M\cong M^*(K_{3,3})$, the result holds. If $r^*(M)=3$, then $M\cong F_7^*$, and thus $M$ has $M(C_4)$ as an induced restriction, a contradiction. If $r^*(M)=4$, then, by a result of Seymour~\cite{seymour} (see also \cite[Lemma 12.2.4]{ox1}), $M$ is isomorphic to either $AG(3,2)$ or $S_8$, and, in each case, $M$ has $M(C_4)$ as an induced restriction, a contradiction. By the Splitter Theorem, if $|E(M)|=9$, then $M$ is an extension of $AG(3,2)$ or $S_8$. Since the rank-4 complements of $AG(3,2)$ and $S_8$ are $F_7$ and $M(K_4)\oplus U_{1,1}$, respectively, there are exactly two possible 9-element matroids as extensions of $AG(3,2)$ and $S_8$, that is, $M$ is the rank-4 complement of $M(K_4)$ or the rank-4 complement of $M(K_4\backslash e)\oplus U_{1,1}$. In each case, $M$ has $M(C_4)$ as an induced restriction, a contradiction.
\end{proof}

To complete the proof of Theorem~\ref{noC4K4}, we only need to consider the matroids with $M^*(K_{3,3})$ as an induced minor.
\begin{lemma}\label{noK33}
    If $M/e$ has $M^*(K_{3,3})$ as an induced minor, then $M$ contains one of $M(C_4),M(K_4),$ or $M^*(K_{3,3})$ as an induced restriction.
\end{lemma}
\begin{proof}
    It suffices to show the result holds when $r(M)=5$ since we may restrict to the relevant rank-5 flat. Assume $M$ does not have any member of $\{M(C_4),M(K_4), M^*(K_{3,3})\}$ as an induced restriction. Let $e_1,e_2,e_3,e_4$ be the standard basis for $P_4$. Then we may assume that the ground set $Z$ of the $M^*(K_{3,3})$-restriction of $M/e$ is $\{e_1,e_2,e_3,e_4,e_1+e_2,e_2+e_3,e_3+e_4,e_1+e_4,e_1+e_2+e_3+e_4\}$. Then the ground set of $M$ may only contain $e$, elements of $Z$, and elements of the form $e+f$ where $f$ is an element of $Z$. Let $T$ be a rank-4 flat of $M$ such that $(M/e)|\cl_{M/e}(T)\cong M^*(K_{3,3})$ and $|T|$ is maximal with this property. Certainly $|T|<9$. Observe that, for every rank-4 flat $F$ of $M$ that avoids $e$, we have $(M/e)|\cl_{M/e}(F)\cong M^*(K_{3,3})$.
    \begin{sublemma}\label{K33T8}
        $|T|<8$.
    \end{sublemma}
    If $|T|=8$, then $M|T$ is isomorphic to a 4-wheel, which has $M(C_4)$ as an induced restriction, a contradiction. Thus \ref{K33T8} holds.

    \begin{sublemma}\label{K33T7}
        $|T|<7$.
    \end{sublemma}
    If $|T|=7$, then, to avoid having $M(C_4)$ as an induced restriction, we may assume $T=\{e_1,e_3,e_4,e_1+e_4,e_2+e_3,e_3~+~e_4,e_1~+~e_2~+~e_3~+~e_4\}$. Therefore, $E(M)$ must contain $\{e,e+e_2,e+e_1+e_2\}$. Look at the flat containing $\{e_3,e_2+e_3,e+e_2,e\}$, noting that $e_2$ is absent. Since this flat is not isomorphic to $M(C_4)$, it must contain either $e+e_3$ or $e+e_2+e_3$. Suppose that $e+e_2+e_3$ is present. Then $M$ contains $\{e_1,e_1+e_4,e_4,e_3+e_4,e_3,e+e_2+e_3,e+e_2,e+e_1+e_2\}$ as eight elements of a rank-4 flat $F$ that avoids $e$. By \ref{K33T8}, $|F|\leq 7$, a contradiction. We deduce that $e+e_2+e_3$ is absent, and hence $e+e_3$ must be present. In this case, $M$ has $\{e_1,e_4,e_1+e_4,e_1+e_2+e_3+e_4,e_2+e_3,e+e_3,e+e_2,e+e_1+e_2\}$ as eight elements of a rank-4 flat that avoids $e$, and again we contradict \ref{K33T8}. 
    We conclude that, \ref{K33T7} holds. 
    
    \begin{sublemma}\label{missingtriangle}
        The set $Z- T$ does not contain a triangle.
    \end{sublemma}
    Without loss of generality, suppose $Z- T$ contains the triangle $\{e_1,e_2,e_1+e_2\}$. Then $M$ must contain the 4-element circuit $\{e,e+e_1,e+e_2,e+e_1+e_2\}$ as a flat, a contradiction. Therefore, \ref{missingtriangle} holds.

    \begin{sublemma}\label{K33T6}
        $|T|<6$.
    \end{sublemma}
    If $|T|=6$, then, since $M$ does not have $M(C_4)$ as an induced restriction and $Z-T$ does not contain a triangle, we may assume $T$ is missing $\{e_2,e_1+e_2,e_2+e_3\}$. Then $M$ contains $\{e+e_2,e+e_1+e_2,e+e_2+e_3\}$. Thus $M$ contains $\{e_1,e+e_1+e_2,e+e_2,e+e_2+e_3,e_3,e_3+e_4,e_4,e_1+e_4\}$ as eight elements of a rank-4 flat of $M$ that avoids $e$, a contradiction to \ref{K33T8}. Therefore, \ref{K33T6} holds.
    \begin{sublemma}\label{K33T5}
        $|T|<5$.
    \end{sublemma}
    If $|T|=5$, then, since $M$ does not have $M(C_4)$ as an induced restriction and $Z-T$ does not contain a triangle, we may assume $T$ is missing one of the following four sets of points.
    \begin{enumerate}[label=(\alph*)]
        \item\label{case2} $\{e_2,e_3,e_3+e_4,e_1+e_2\}$
        \item\label{case3} $\{e_2,e_2+e_3,e_3+e_4,e_1+e_2\}$
        \item\label{case4} $\{e_1,e_2,e_3+e_4,e_1+e_2+e_3+e_4\}$
    \end{enumerate}
    In case \ref{case2}, $M$ has $\{e_1,e+e_1+e_2,e+e_2,e+e_3,e_2+e_3,e_1+e_2+e_3+e_4,e_1+e_4,e_4\}$ as eight elements of a rank-4 flat of $M$ that avoids $e$, which contradicts \ref{K33T8}. In case \ref{case3}, $\{e_1,e+e_1+e_2,e+e_2,e+e_2+e_3,e_3,e_4,e_1+e_4\}$ spans a rank-4 flat $F$ of $M$ that avoids $e$ and has at least seven elements, which contradicts \ref{K33T7}. In case \ref{case4}, $M$ has $\{e+e_2,e+e_2+e_3,e_3,e_4,e+e_1,e_1+e_2\}$ contained in a rank-4 flat of $M$ avoiding $e$, which contradicts \ref{K33T6}. Therefore, \ref{K33T5} holds.

    If $|T|=4$, then, as $T$ must be independent and $Z-T$ contains no triangles, we may assume that $T=\{e_1,e_1+e_2,e_3,e_1+e_2+e_3+e_4\}$. This implies that $M$ has $\{e_1,e+e_2,e+e_2+e_3,e_3,e+e_3+e_4\}$ spanning a rank-4 flat of $M$ that avoids $e$, a contradiction to \ref{K33T5}. 
\end{proof}

The next result establishes a connection between the excluded induced minors of a class of matroids and the excluded induced restrictions of that class of matroids.
\begin{lemma}\label{ir=im}
Suppose $\mathcal{N}$ is a class of matroids such that if a matroid $M$ has an element $e$ such that $M/e$ is isomorphic to a member of $\mathcal{N}$, then $M$ has a member of $\mathcal{N}$ as an induced restriction. Let $\mathcal{N}'$ be the class of induced-minor-minimal members of $\mathcal{N}$. Then the class of matroids with no member of $\mathcal{N}'$ as an induced minor coincides with the class of matroids with no member of $\mathcal{N}$ as an induced restriction.
\end{lemma}
\begin{proof}
    Clearly if $M$ has a member of $\mathcal{N}$ as an induced restriction, then $M$ has a member of $\mathcal{N}'$ as an induced minor. Conversely, suppose $M$ has a member of $\mathcal{N}'$ as an induced minor. We may assume that $M$ has no member of $\mathcal{N}'$ as an induced restriction. Then $M$ has a flat $F$ and a nonempty independent set $T$ such that $(M|F)/T$ is isomorphic to a member of $\mathcal{N}'$ and hence a member of $\mathcal{N}$. We argue by induction on $|T|$ that $M$ has a member of $\mathcal{N}$ as an induced restriction. 

    When $|T|=1$, we see that the assertion holds by the definition of $\mathcal{N}$. Now assume the result holds for $|T|<n$, and let $|T|=n\geq 2$. Let $M_1=M|F$, and take $e$ in $T$. Then $(M_1/e)/(T-e)$ is isomorphic to a member of $\mathcal{N}$. Thus, by the induction assumption, $M_1/e$ has a member of $\mathcal{N}$ as an induced restriction. Then, by the induction assumption again, $M_1$ has a member of $\mathcal{N}$ as an induced restriction, and hence $M$ has a member of $\mathcal{N}$ as an induced restriction.
\end{proof}

\begin{proof}[Proof of Theorem \ref{noC4K4}]
Let $\mathcal{N}=\{M(C_n):n\geq 4\}\cup \{M(K_4),M^*(K_{3,3})\}$. Then, by Lemmas~\ref{circRes},~\ref{K4incontraction},~and~\ref{noK33}, the set $\mathcal{N}$ has the property that if $M/e$ is isomorphic to a member of $\mathcal{N}$, then $M$ has a member of $\mathcal{N}$ as an induced minor. Since $\{M(C_4),M(K_4)\}$ is the set of induced-minor-minimal members of $\mathcal{N}$, the result holds by Lemma~\ref{ir=im}.
\end{proof}

Next we prove the analog of Theorem~\ref{noC4K4} when $q>2$. Although the result is rather unattractive, we include it for completeness. The equivalence of (i) and (ii) was shown in \cite{douthox}. Because $U_{3,q+2}$ is $GF(q)$-representable if and only if $q$ is even, when $q$ is odd, we can omit $U_{3,q+2}$ from (ii). This corrects a small error in Theorem 1.4 of \cite{douthox}. Note that some matroids in (iii) may be excluded as they fail to be $GF(q)$-representable. But the problem of determining precisely which uniform matroids are $GF(q)$-representable is open. For results on exactly which uniform matroids are known to be $GF(q)$-representable, we direct the reader to \cite{Ball, BallDeBeule, HirschStorme, HirschThas} (see also \cite[Conjecture 6.5.20]{ox1}). Recall that the class of $GF(q)$-chordal matroids is the class of matroids that can be obtained by a sequence of generalized parallel connections across projective geometries over $GF(q)$ starting with projective geometries over $GF(q)$.

\begin{theorem}\label{GF(q)ir}
    When $q>2$, the following are equivalent for a $GF(q)$-representable matroid $M$.
    \begin{enumerate}[label=\normalfont(\roman*)]
        \item $M$ is $GF(q)$-chordal.
        \item $M$ has no member of $\{U_{2,k}:2<k\leq q\}\cup\{U_{3,q+2}\}$ as an induced minor.
        \item $M$ has no member of $\{U_{n,n+1}:n\geq 2\}\cup \{U_{2+t,k+t}: 4\leq k\leq q \text{ and }0\leq t\leq q-3 \}\cup \{U_{3,q+2}\}$ as an induced restriction.
    \end{enumerate}
\end{theorem}
\begin{lemma}\label{noU2k}
    Suppose $M$ is $GF(q)$-representable. If $M/e$ has $U_{2,n}$ as an induced minor for some $n$ with $3\leq n\leq q$, then $M$ has a member of $\{U_{2,k}: 3\leq k\leq q\}\cup \{U_{3,n+1}\}$ as an induced restriction.
\end{lemma}
\begin{proof}
    It suffices to show that the result holds for $r(M)=3$ since we may restrict to the relevant rank-3 flat. Let $E(M/e)=\{x_1,x_2,\dots,x_k\}$. If $M$ has a $(q+1)$-point line, then $e$ must be on that line, otherwise, $M/e$ would be $U_{2,q+1}$. Without loss of generality, suppose $\cl(\{e,x_k\})$ is a $(q+1)$-point line, $L_1$. Then the line $L_2=\cl(\{x_1,x_2\})$ must intersect $L_1$ at a point different from $e$. Therefore, $L_2$ is either isomorphic to $U_{2,k}$ for some $k$ with $3\leq k\leq q$, in which case the result holds, or $L_2$ is a $(q+1)$-point line, and $M/e$ is isomorphic to $U_{2,q+1}$, a contradiction. Thus, we may assume that $M$ has no $(q+1)$-point lines. Let $X$ be a 3-element subset of $E(M)$. If $X$ is dependent, then, since $M$ is simple having no $(q+1)$-point lines, $M|\cl(X)$ is isomorphic to a member of $\{U_{2,k}: 3\leq k\leq q\}$ and the result holds. We deduce that $X$ is independent. Therefore, $M$ is isomorphic to $U_{3,n+1}$.
\end{proof}

To complete the proof Theorem~\ref{GF(q)ir}, we consider the following general result for uniform matroids.
\begin{lemma}\label{noUrn}
    Let $M$ be $GF(q)$-representable for some $q>2$ and suppose $M/e\cong U_{r,n}$ for some $r$ and $n$ with $2<r<n$. Then $M$ has a member of $\{U_{2+t,k+t}: 3\leq k\leq q \text{ and } 0\leq t\leq q-3\}\cup  \{U_{r,
n}, U_{r+1,n+1}\}$ as an induced restriction. 
\end{lemma}
\begin{proof}
    Assume the result fails. Observe that, since $U_{2,q+2}$ is not $GF(q)$-representable, we must have $n\leq q+r-1$. Next note that if $M$ has a hyperplane $H_0$ that avoids $e$, then $M|H_0$ must be isomorphic to a uniform matroid for if  $H_0$ contains a non-spanning circuit $C$, then $M/e$ must also contain $C$ as a non-spanning circuit. Moreover, $H_0$ can contain at most $n$ elements as $|E(M/e)|=n$. Therefore, $M|H_0$ is isomorphic to $U_{r,s}$ for some $s$ with $r\leq s \leq n$. Suppose that $H_0$ can be chosen so that $s>r$. If $s=n$, then $M$ has $U_{r,n}$ as an induced restriction, a contradiction. Thus we may suppose that $s<n$. Observe that if $r>q-1$, then, since $U_{q,q+2}$ is not $GF(q)$-representable, we must have $n=r+1$, so $s=n$, a contradiction. Thus we may assume that $r\leq q-1$. Then we obtain a contradiction by taking $t=r-2$ and $k=s-r+2$ for then $t\leq q-3$ and $k<n-r+2\leq (q+r-1)-r+2=q+1$.
    Hence we may assume that every hyperplane that avoids $e$ is an independent set of size $r$. 
    
    Let $E(M/e)=\{x_1,x_2,\dots,x_n\}$. 
    If $M$ has a $(q+1)$-point line, this line must contain $e$ as $M/e$ does not have a $(q+1)$-point line. Without loss of generality, suppose $\cl_M(\{e,x_n\})$ is a $(q+1)$-point line, $L$. Then $\cl_M(\{x_1,x_2,\dots,x_r\})$ is a hyperplane $H$ of $M$ that meets $L$ at a point different from $e$. Therefore $H$ is a hyperplane that avoids $e$ and contains a circuit of $M$, a contradiction. We conclude that $M$ has no $(q+1)$-point lines. Therefore, either $e$ is contained in a $U_{2,k}$ induced restriction for some $k$ with $3\leq k\leq q$, and the result holds, or every circuit containing $e$ has size $r+2$. Let $X$ be a subset of $E(M)-e$. If $|X|<r+1$, then $X$ is certainly independent as $X$ must be independent in $M/e$. Suppose $|X|=r+1$ and that $X$ is dependent. Then $\cl_M(X)$ is a hyperplane that avoids $e$ and contains a circuit of $M$, a contradiction. Therefore, $X$ is an independent set of size $r+1$, and $M$ is isomorphic to $U_{r+1,n+1}$. 
\end{proof}

\begin{proof}[Proof of Theorem~\ref{GF(q)ir}] 
It suffices to show the equivalence of (ii) and (iii).   
Let $\mathcal{N}=\{U_{n,n+1}:n\geq 2\}\cup \{U_{2+t,k+t}: 4\leq k\leq q \text{ and } 0\leq t \leq q-3\}\cup\{U_{3,q+2}\}$. Then, since $U_{4,q+3}$ is not $GF(q)$-representable for any $q$, it follows by Lemmas~\ref{noU2k}~and~\ref{noUrn} that $\mathcal{N}$ has the property that if $M/e$ is isomorphic to a member of $\mathcal{N}$, then $M$ has a member of $\mathcal{N}$ as an induced minor. Since $\{U_{2,k}:3\leq k\leq q\}\cup \{U_{3,q+2}\}$ is the set of induced-minor-minimal members of $\mathcal{N}$, the result holds by Lemma~\ref{ir=im}.

Note if $M$ has $U_{2+t,k+t}$ as an induced restriction for some $t$ and $k$ with $t\geq q-2$ and $k\geq 4$, then $M$ has $U_{q,q+2}$ as a minor, so $M$ is not $GF(q)$-representable, a contradiction. 
\end{proof}

Theorem~\ref{gf(p)-chord} follows from Theorem~\ref{GF(q)ir} and a result of Ball \cite{Ball} (see also \cite[Theorem 6.5.21]{ox1}), which establishes precisely when a uniform matroid is $GF(p)$-representable for $p$ a prime.

When $p=3$, Theorem~\ref{gf(p)-chord} gives the following.

\begin{corollary}
    The following are equivalent for a ternary matroid $M$.
    \begin{enumerate}[label=\normalfont(\roman*)]
        \item $M$ is $GF(3)$-chordal.
        \item $M$ has no $M(C_3)$  induced minor.
        \item $M$ has no member of $\{M(C_n): n\geq 3\}$ as an induced restriction.
    \end{enumerate}
\end{corollary}

Similarly, Theorem~\ref{GF(q)ir} gives the following when $q=4$.

\begin{corollary}
    The following are equivalent for a matroid $M$ that is representable over $GF(4)$.
    \begin{enumerate}[label=\normalfont(\roman*)]
        \item $M$ is $GF(4)$-chordal.
        \item $M$ has no member of $\{U_{2,3},U_{2,4},U_{3,6}\}$ as an induced minor.
        \item $M$ has no member of $\{U_{n,n+1}:n\geq 2\}\cup \{U_{2,4},U_{3,5},U_{3,6}\}$ as an induced restriction.
    \end{enumerate}
\end{corollary}

\section{Perfect Elimination Ordering}\label{Perfect Elimination Ordering}

In this section, we will prove Theorem~\ref{perfectisgfq}. Recall that a perfect elimination ordering of cocircuits of a matroid $M$ is a collection of sets $C_1^*,C_2^*,\dots,C_r^*$ such that, for all $i$ in $[r]$, the set $C_i^*$ is a cocircuit of the matroid $M_i=M\backslash (C_1^*\cup C_2^*\cup \dots \cup C_{i-1}^*)$ and $M|\cl_{M_i}(C_i^*)$ is a projective geometry. Observe that $M|\cl_{M_i}(C_i^*)=M_i|\cl_{M_i}(C_i^*)$ and that $E(M_i)$ is a flat of $M$ of rank $r(M)-i+1$. We omit the straightforward proof of the next result.

\begin{lemma}\label{perfectPG}
    Projective geometries have a perfect elimination ordering of cocircuits.
\end{lemma}

We now prove the second main result of the paper.

\begin{proof}[Proof of Theorem \ref{perfectisgfq}]
    Suppose $M$ is a $GF(q)$-chordal matroid that does not have a perfect elimination ordering of cocircuits and has $|E(M)|$ a minimum among such matroids. Since $M$ is $GF(q)$-chordal, $M$ may be written as a $P_N(M_1,M_2)$ where $N$ is a projective geometry and $M_2$ is a projective geometry. Therefore, $M_2$ has a perfect elimination ordering of cocircuits. Choose $C_1^*$ to be a cocircuit of $M_2$ that avoids $E(N)$. By the minimality of $M$, the matroid $M\backslash C_1^*$ has a perfect elimination ordering of cocircuits $C_2^*,C_3^*,\dots,C_r^*$. Therefore, $M$ has a perfect elimination ordering.

    Suppose $M$ has a perfect elimination ordering of cocircuits labeled $C_1^*,C_2^*,\dots,C_r^*$. Let $M_1=M$, and let $M_i=M\backslash (C_1^*\cup C_2^*\cup \dots \cup C_{i-1}^*)$ for each $i$ with $2\leq i\leq r$. Certainly, $M_r$ is a projective geometry, and therefore is $GF(q)$-chordal. Let $k$ be the smallest integer such that $M_k$ is not $GF(q)$-chordal. Then $M_{k+1}$ is $GF(q)$-chordal, $M_k|\cl_{M_k}(C_k^*)$ is a projective geometry, and $M_{k+1}$ is a hyperplane of $M_k$. Thus $E(M_{k+1})\cap E(M_k|\cl_{M_k}(C_k^*))$ is a flat, $N$, of a projective geometry and so $M|N$ is a projective geometry. Hence $M_k$ is a generalized parallel connection of $GF(q)$-chordal matroids across a projective geometry, namely $P_{M|N}(M_{k+1},M|\cl_{M_k}(C_k^*))$. Thus, $M_k$ is $GF(q)$-chordal, a contradiction.
\end{proof}

On combining Theorems~\ref{noC4K4} and \ref{perfectisgfq}, we obtain the following.

\begin{corollary}
    The following are equivalent for a binary matroid $M$.
    \begin{enumerate}[label=\normalfont(\roman*)]
        \item $M$ is $GF(2)$-chordal.
        \item $M$ has no member of $\{M(C_4),M(K_4)\}$ as an induced minor.
        \item $M$ has no member of $\{M(C_n) : n\geq 4\}\cup \{M(K_4),M^*(K_{3,3})\}$ as an induced restriction.
        \item $M$ has a perfect elimination ordering of cocircuits.
    \end{enumerate}
\end{corollary}

\section*{Acknowledgments}
The authors thank Tuong Le and Chayim Lowen for pointing out the error in Lemma 2.5 of the published version of this paper.

\end{document}